\documentclass[10pt]{amsart}
\usepackage{amssymb}
\usepackage{amsmath}
\usepackage{amsfonts}
\usepackage{graphicx}
\usepackage{color,xcolor}

\newtheorem{thm}{Theorem}[section]
\newtheorem{cor}[thm]{Corollary}
\newtheorem{lem}[thm]{Lemma}

\newtheorem{prop}[thm]{Proposition}

\newtheorem{rem}[thm]{Remark}

\numberwithin{equation}{section}
\begin{document}

\title{Digraphs from Endomorphisms of Finite Cyclic Groups}

\author{Min Sha}
\address{Institut de Mathematiques de Bordeaux, Universite Bordeaux 1,
 351, cours de la Liberation, 33405 Talence Cedex, France}
\email{shamin2010@gmail.com}


\subjclass[2010]{Primary 05C05, 05C20; Secondary 05C25, 05C50}



\keywords{cyclic group, digraph, tree, adjacency matrix, automorphism group}

\begin{abstract}
We associate each endomorphism of a finite cyclic group with a digraph and study many properties of this digraph,
including its adjacency matrix and automorphism group.
\end{abstract}

\maketitle



\section{Introduction}

As we all know, we can construct Cayley graphs and Cayley digraphs from a group, and these graphs are vertex-transitive.
In this article, we construct digraphs from a finite cyclic group by using its endomorphisms. In general,
these digraphs are not vertex-transitive. But they have many good properties which may make them into beautiful graphs and may
merit further researches.

Let $H$ be a finite cyclic group with $n$ elements, $n>1$, we treat it as a multiplicative group.
We denote its identity element by $1$ without confusion.
 As we all know, $H$ has $n$ endomorphisms, every endomorphism has a unique form $f:H \to H,x\to x^k,k\in \mathbb{Z}, 1\le k\le n$,
and $f$ is an isomorphism if and only if $n$ and $k$ are coprime.
We can consider the digraph that has the elements of $H$
as vertices and a directed edge from $a$ to $b$ if and only if $f(a)=b$.
Since cyclic groups with the same order are isomorphic, this digraph only depends on $n$ and $k$.
So we can denote this digraph by $G(n,k)$. For example, see Figure 1 in section 4.
\cite{B} studied the digraph from any endomorphism of $\mathbb{Z}/n\mathbb{Z}$,
especially the author studied the number of cycles.
\cite{BHM}, \cite{R} and \cite{VS} studied the digraph from the endomorphism $f(x)=x^2$ of $(\mathbb{Z}/p\mathbb{Z})^{*}$, $p$ is a prime.
In particular, the cycle and tree structures have been classified.
\cite{LMR} generalized those results in \cite{BHM} to the digraph from any endomorphism of $(\mathbb{Z}/p\mathbb{Z})^{*}$.
\cite{SH} studied some elementary properties of the digraph from any endomorphism of $(\mathbb{F}_{q})^{*}$,
$\mathbb{F}_{q}$ is a finite field with $q$ elements.

In section 2 and section 3, we generalize those results in \cite{LMR} to $G(n,k)$, and we consider many other properties of $G(n,k)$.
In section 4 and section 5, we consider its adjacency matrix and automorphism group respectively,
furthermore we determine its characteristic polynomial and minimal polynomial.

Especially, the results here may have applications to monomial dynamical systems over finite fields, see \cite{SH}.

\section{Basic Properties of $G(n,k)$}

Given two integers $l$ and $m$, we denote their greatest common divisor and least common multiple
by $(l,m)$ and $[l,m]$ respectively.

First we factor $n$ as $tw$ where $t$ is the largest factor of $n$ relatively prime to $k$.
So $(k,t)=1$ and $(w,t)=1$. For any $a\in H$, Let ord$(a)$ denote its order.

 For proceeding further, we need the following lemma.
\begin{lem}\label{solution}
Let $a\in H$, the equation $x^{k}=a$ has a solution if and only if $a^{\frac{n}{d}}=1$, $d=(n,k)$.
Moreover, if the equation has a solution, it has exactly $d$ solutions.
\end{lem}
\begin{proof}
Applying the same argument as Proposition 7.1.2 in \cite{IR}.
\end{proof}

The following lemma is easy to prove but fundamental to the understanding of the structure of $G(n,k)$.
We omit its proof and refer the readers to \cite{LMR}.
\begin{lem}\label{basic}
We have the following elementary properties of $G(n,k)$.

{\rm(1)} The outdegree of any vertex in $G(n,k)$ is $1$.

{\rm(2)} The indegree of any vertex in $G(n,k)$ is $0$ or $(n,k)$. Moreover, the indegree of $a\in G$ is $(n,k)$ if and only if $a^{\frac{n}{(n,k)}}=1$.

{\rm(3)} $G(n,k)$ has $n$ vertices and $n$ directed edges.

{\rm(4)} Given $a,b\in H$, there exists a directed path from $a$ to $b$ if and only if there exists a positive integer $m$ such that $a^{k^{m}}=b$.

{\rm(5)} Given any element in $G(n,k)$, repeated iteration of $f$ will eventually lead to a cycle.

{\rm(6)} Every component of $G(n,k)$ contains exactly one cycle.

{\rm(7)} The set of non-cycle vertices forms a forest.
\end{lem}

\begin{prop}\label{indegree}
The number of the vertices with indegree $0$ is $\frac{d-1}{d}n$, where $d=(n,k)$.
\end{prop}
\begin{proof}
By Lemma \ref{solution}, a vertex $a$ has non-zero indegree if and only if $a^{\frac{n}{d}}=1$.
Hence, the vertices with non-zero indegree form a subset $H_{d}=\{x\in H|x^{\frac{n}{d}}=1\}$.
It is well-known that $H_{d}$ is a cyclic subgroup of $H$ with $\frac{n}{d}$ elements. So we get the desired result.
\end{proof}

As follows, we want to study the cycle structures of $G(n,k)$.

\begin{prop}\label{vertex}
The vertex $a$ is a cycle vertex if and only if ${\rm ord}(a)|t$.
\end{prop}
\begin{proof}
Suppose $a$ is a cycle vertex. Then there exists a positive integer $m$ such that $a^{k^{m}}=a$.
So ${\rm ord}(a)|(k^{m}-1)$, which implies $({\rm ord}(a),k)=1$. So $({\rm ord}(a),w)=1$.
Note that ${\rm ord}(a)|n$, then ${\rm ord}(a)|t$.

Conversely, suppose ${\rm ord}(a)|t$. Then $({\rm ord}(a),k)=1$. So there exists a positive integer $m$
such that ${\rm ord}(a)|(k^{m}-1)$, which implies $a^{k^{m}}=a$. So $a$ is a cycle vertex.
\end{proof}

\begin{cor}
There are exactly $t$ cycle vertices in $G(n,k)$.
\end{cor}
\begin{proof}
From Proposition \ref{vertex}, the total number of cycle vertices is $\sum\limits_{d|t}\varphi(d)=t$,
where $\varphi$ is the Euler's $\varphi$-function, and $\varphi(d)$ is the number of elements with order $d$.
\end{proof}

\begin{prop}\label{same order}
Vertices in the same cycle have the same order.
\end{prop}
\begin{proof}
Assume $a$ and $b$ are in the same cycle. So there exists a $m$ such that $a^{k^{m}}=b$, which implies $b^{{\rm ord}(a)}=1$.
So ${\rm ord}(b)|{\rm ord}(a)$. Similarly, we have ${\rm ord}(a)|{\rm ord}(b)$. So ${\rm ord}(a)={\rm ord}(b)$.
\end{proof}

By Proposition \ref{same order}, the notion of the order of a cycle is well-defined.
Let $\ell(d)$ denote the length of a cycle with order $d$, where $d|t$. If two integers $l$ and $m$ are coprime,
let ${\rm ord}_{l}m$ denote the exponent of $m$ modulo $l$.
\begin{prop}\label{cycle}
Let $d$ and $r$ be orders of cycles. Then:

$(1)$ $\ell(d)={\rm ord}_{d}k$.

$(2)$ The longest cycle length in $G(n,k)$ is $\ell(t)={\rm ord}_{t}k$.

$(3)$ There are $\varphi(d)/\ell(d)$ cycles of order $d$.

$(4)$ The total number of cycles in $G(n,k)$ is $\sum\limits_{d|t}\frac{\varphi(d)}{\ell(d)}$.

$(5)$ $\ell([d,r])=[\ell(d),\ell(r)]$.
\end{prop}
\begin{proof}
(1) Let $a$ be a vertex in a cycle of order $d$. It is obvious that $\ell(d)$ is the smallest positive integer
such that $a^{k^{\ell(d)}}=a$, that is the smallest positive integer such that $d|(k^{\ell(d)}-1)$. So $\ell(d)={\rm ord}_{d}k$.

(2) By (1) and Proposition \ref{vertex}.

(3) Notice that the number of elements with order $d$ is $\varphi(d)$.

(4) By (3) and Proposition \ref{vertex}.

(5) Since $d|[d,r]$, $\ell(d)|\ell([d,r])$. Similarly, we have $\ell(r)|\ell([d,r])$.
So $[\ell(d),\ell(r)]|\ell([d,r])$. In addition, since $d|(k^{\ell(d)}-1)$, $d|(k^{[\ell(d),\ell(r)]}-1)$.
Similarly, $r|(k^{[\ell(d),\ell(r)]}-1)$. So $[d,r]|(k^{[\ell(d),\ell(r)]}-1)$. Hence, $\ell([d,r])|[\ell(d),\ell(r)]$.
So we have $\ell([d,r])=[\ell(d),\ell(r)]$.
\end{proof}

But $\ell((d,r))=(\ell(d),\ell(r))$ is not always true. For example, let $k=2, d=11$ and $r=15$,
we have $(11,15)=1$ and $\ell(1)=1$, but $(\ell(11),\ell(15))=(10,4)=2$.

\begin{rem}
{\rm Let $\mu$ be M$\ddot{\rm o}$bius function. Similar as Proposition 2.5 in \cite{SH}, the number of cycles with length $r$ is
$\frac{1}{r}\sum\limits_{d|r}\mu(d)(k^{r/d}-1,n)$.}
\end{rem}

\begin{cor}
If a component has a generator of $H$, then its unique cycle has the longest length $\ell(t)$.
\end{cor}
\begin{proof}
Since if a component has a generator of $H$, the order of its unique cycle is $t$.
\end{proof}

\begin{prop}
 Every generator of $H$ has indegree $0$ if and only if $(n,k)\ne 1$.
\end{prop}
\begin{proof}
Suppose $(n,k)\ne 1$. For any generator $b$ of $H$, if the indegree of $b$ is not 0, then there exists a vertex $a$ such that $a^{k}=b$.
Since $n={\rm ord}(b)=\frac{{\rm ord}(a)}{({\rm ord}(a),k)}$ and ${\rm ord}(a)|n$,
${\rm ord}(a)=n$ and $(n,k)=1$. This leads to a contradiction.

Conversely, if every generator of $H$ with indegree $0$, then generators are not cycle vertices.
By Proposition \ref{vertex}, $t\ne n$. So $(n,k)\ne 1$.
\end{proof}

Hence, if $(n,k)\ne 1$, since $H$ has $\varphi(n)$ generators, by Proposition \ref{indegree},
we have $\varphi(n)\le \frac{d-1}{d}n$, where $d=(n,k)$.

Now we would like to consider which kind of graphs $G(n,k)$ belongs to.
\begin{prop}\label{regular}
The following statements are equivalent.

$(1)$ $G(n,k)$ is regular of degree $1$.

$(2)$ Every component of $G(n,k)$ is a cycle.

$(3)$ $f$ is an automorphism.
\end{prop}
\begin{proof}
Note that $f$ is an automorphism if and only if $(n,k)=1$, then applying Lemma \ref{basic} (2) and (7).
\end{proof}

\begin{prop}\label{connected}
$G(n,k)$ is connected if and only if there exists a positive integer $m$ such that $n|k^{m}$.
\end{prop}
\begin{proof}
Suppose $n|k^{m}$. Then for any $a\in H$, $a^{k^{m}}=1$. So $G(n,k)$ is connected.

Conversely, suppose $G(n,k)$ is connected. By Lemma \ref{basic} (7), there is only one cycle, that is $\{1\}$.
By Lemma \ref{basic} (6), for any $a\in H$, there is a positive integer $m$ such that $a^{k^{m}}=1$.
If $a$ is a generator of $H$, then $n|k^{m}$.
\end{proof}

Notice that there exists a positive integer $m$ such that $n|k^{m}$ if and only if $t=1$. Hence,
$G(n,k)$ is connected if and only if $G(n,k)$ has only one cycle vertex, that is the identity element.

\begin{prop}
The following statements are equivalent.

$(1)$ $G(n,k)$ is arc-transitive.

$(2)$ $G(n,k)$ is vertex-transitive.

$(3)$ $f$ is the identity.
\end{prop}
\begin{proof}
Note that there exist loops in $G(n,k)$. So $G(n,k)$ is arc-transitive if and only if there are no other edges except loops,
that is for any $a\in G(n,k), f(a)=a$, that is for any $a\in G(n,k), a^{k-1}=1$, that is $n|(k-1)$.

Applying the same argument as the above paragraph, we have $G(n,k)$ is vertex-transitive if and only if $n|k-1$.

Notice that $1\le k\le n$, we get the desired result.
\end{proof}

Since the number of distinct endomorphisms of $H$ is $n$, we attain $n$ distinct digraphs by our manner.
There is an interesting problem that whether there exist isomorphic digraphs among them. In \cite{LMR}, the authors
gave an example $G(10,2)\cong G(10,8)$.

\begin{prop}
If $n$ is a prime, for any $1<k_{1}<k_{2}<n$, $G(n,k_{1})\cong G(n,k_{2})$ if and only if ${\rm ord}_{n}k_{1}={\rm ord}_{n}k_{2}$.
\end{prop}
\begin{proof}
Since $(n,k_{1})=1$, by Proposition \ref{regular}, each component of $G(n,k_{1})$ is a cycle.
By Proposition \ref{cycle} (1), there are only two kinds of cycles in $G(n,k_{1})$, one with length $1$,
the other with length ${\rm ord}_{n}k_{1}$. Since $(n,k_{1}-1)=1$, there is only one cycle with length $1$.
By Proposition \ref{cycle} (3), there are $\frac{n-1}{{\rm ord}_{n}k_{1}}$ cycles with length ${\rm ord}_{n}k_{1}$.

We can get similar results for $G(n,k_{2})$. Then we can get the desired result.
\end{proof}

\section{Properties of Trees}

Here we introduce some notations for the tree originating from any given cycle vertex.

For $m\ge 1$, we say a non-cycle vertex $a$ has height $m$ with respect to a cycle vertex $c$
if $m$ is the smallest positive integer such that $a^{k^{m}}=c$.
For $m\ge 1$, let $T_{c}^{m}$ denote the set of non-cycle vertices with height $m$ with respect to the cycle vertex $c$.
Similarly, $T^{m}$ denotes the set of all vertices with height $m$.
For convenience, we put $T_{c}^{0}=\{c\}$ and say $c$ has height 0, $T^{0}$ denotes the set of all cycle vertices.
Let $F_{c}$ be the induced subgraph of $G(n,k)$ with vertices $\bigcup_{m\ge1}T_{c}^{m}$.
In fact, $F_{c}$ is a forest if it is not empty. We can get an induced subgraph of $G(n,k)$ with vertices $\bigcup_{m\ge0}T_{c}^{m}$,
and we delete the loop if it exists, then we get a tree and denote it by $T_{c}$.

All the vertices lie in the trees we define above. As follows, without special instructions,
 the concept of tree means what we define in the above.

We will show that for any cycle vertex $c$, $T_{c}\cong T_{1}$.

\begin{lem}\label{product}
The product of a non-cycle vertex and a cycle vertex is a non-cycle vertex.
\end{lem}
\begin{proof}
Notice that by Proposition \ref{vertex}, the cycle vertices of $G(n,k)$ form a subgroup.
\end{proof}

\begin{lem}\label{product1}
If $a\in T_{1}^{h}, h\ge 1$ and $c$ is a cycle vertex, then $ac\in T_{c^{k^{h}}}^{h}$.
\end{lem}
\begin{proof}
By Lemma \ref{product}, $ac\notin T^{0}$. Furthermore, $(ac)^{k^{h}}=c^{k^{h}}$ is a cycle vertex
but $(ac)^{k^{h-1}}$ is a non-cycle vertex because $a^{k^{h-1}}\notin T^{0}$, which implies $ac\in T_{c^{k^{h}}}^{h}$.
\end{proof}

\begin{thm}\label{iso1}
Let $c$ be a cycle vertex, then $F_{c}\cong F_{1}$.
\end{thm}
\begin{proof}
First we show that there exists an one to one correspondence between the vertices of $T_{1}^{h}$ and $T_{c}^{h}$ for all heights $h\ge 1$,
and hence between $F_{1}$ and $F_{c}$. Let $h$ be fixed and let $c_{h}$ denote the unique cycle vertex such that
$c_{h}^{k^{h}}=c$. From Lemma \ref{product1}, define $g_{h}:T_{1}^{h}\to T_{c}^{h}$ by $g_{h}(a)=ac_{h}$.

For any $b\in T_{c}^{h}$, $(b\cdot c_{h}^{-1})^{k^{h}}=b^{k^{h}}c^{-1}=1$ and $(b\cdot c_{h}^{-1})^{k^{h-1}}\notin T^{0}$
because $b^{k^{h-1}}\notin T^{0}$. It follows that $b\cdot c_{h}^{-1}\in T_{1}^{h}$. Then $g_{h}(b\cdot c_{h}^{-1})=b$.
So $g_{h}$ is surjective. It is obvious that $g_{h}$ is injective. So $g_{h}$ is one to one.

Combining these $g_{h}$, we get a bijective map $g$ from $F_{1}$ to $F_{c}$.

It remain to show that $g$ is indeed an isomorphism.
For any directed edge of $F_{1}$, it is from some $a\in T_{1}^{h}$ to $a^{k}\in T_{1}^{h-1}$ for some $h$.
We only need to show that there exists a directed edge from $g(a)$ to $g(a^{k})$ in $F_{c}$, that is $(g(a))^{k}=g(a^{k})$,
that is $(g_{h}(a))^{k}=g_{h-1}(a^{k})$. Now $c_{h}^{k^{h}}=c$ implies $(c_{h}^{k})^{k^{h-1}}=c$,
by the uniqueness of $c_{h-1}$, we have $c_{h}^{k}=c_{h-1}$. So $(g_{h}(a))^{k}=(ac_{h})^{k}=a^{k}c_{h-1}=g_{h-1}(a^{k})$.
\end{proof}

\begin{cor}\label{iso}
Let $c$ be a cycle vertex, then $T_{c}\cong T_{1}$.
\end{cor}
\begin{proof}
Applying Theorem \ref{iso1} and the relation between $T_{c}$ and $F_{c}$.
\end{proof}

Hence, every tree has the same height, denote it by $h_{0}$, and different trees have the same number of vertices in each height.

\begin{cor}
For any two components $G_{1}$ and $G_{2}$ of $G(n,k)$, $G_{1}\cong G_{2}$ if and only if
the unique cycles in them have the same length.
\end{cor}

There is another property of the map $g_{h}$ in Theorem \ref{iso1}, see the following proposition.
\begin{prop}
If $a\in T_{1}^{h}$ and $b\in T_{c}^{h}$ with $c_{h}$ the cycle vertex such that $b=ac_{h}$,
then ${\rm ord}(b)={\rm ord}(a)\cdot {\rm ord}(c)$.
\end{prop}
\begin{proof}
Since $a^{k^{h}}=1$, ${\rm ord}(a)|k^{h}$. By Proposition \ref{same order}, ${\rm ord}(c_{h})={\rm ord}(c)|t$. So $({\rm ord}(a), {\rm ord}(c_{h}))=1$. It follows that ${\rm ord}(b)={\rm ord}(a)\cdot {\rm ord}(c)$.
\end{proof}

As follows, we would like to study the tree structures by using heights.

For any $a\in H$, denote its order ${\rm ord}(a)$ by $n_{a}$, and factor $n_{a}$ by $t_{a}w_{a}$, where $t_{a}$ is the
largest factor of $n_{a}$ relatively prime to $k$. So $n_{a}|n, t_{a}|t$ and $w_{a}|w$. Similarly, we denote $a$'s height by $h_{a}$.
The next proposition shows that $h_{a}$ only depends on $w_{a}$.

\begin{prop} \label{height1}
For any $a\in H$, $h_{a}$ is the minimal $h$ such that $w_{a}|k^{h}$.
Especially, $h_{0}$ is the minimal $h$ such that $w|k^{h}$.
\end{prop}
\begin{proof}
If $n_{a}\mid t$, then $a$ is a cycle vertex. So $h_{a}=0$. Note that $w_{a}=1$, so the conclusion is correct in this case.

If $n_{a}\nmid t$. Since $h_{a}$ is the minimal $h$ such that $a^{k^{h}}$ is a cycle vertex,
that is the minimal $h$ such that ${\rm ord}(a^{k^{h}})=\frac{n_{a}}{(n_{a},k^{h})}\mid t$, then $h_{a}$ is the minimal $h$ such that
$(n_{a},k^{h})=w_{a}$, that is the minimal $h$ such that $w_{a}\mid k^{h}$.
\end{proof}

\begin{cor}\label{height2}
For any two vertices $a$ and $b$,if $w_{a}=w_{b}$, then they have the same height. Especially, The vertices with the same order
are at the same height.
\end{cor}

But if $a$ and $b$ have the same height, maybe $w_{a}\ne w_{b}$. For example, see Figure 3 in section 5,
let $a=9$ and $b=40$, then $a$ and $b$ have the same height, but $w_{a}=4$ and $w_{b}=2$.

\begin{cor}\label{largest}
For any vertex $a$, if $w_{a}=w$, then $a$ is at the largest height. Especially,
the generators of $H$ must be at the largest height.
\end{cor}
\begin{proof}
For any vertex $a$, $w_{a}\mid w$, then applying Proposition \ref{height1}, we get the desired result.
\end{proof}

\begin{cor} \label{prime}
If $k$ is a prime, then for any two vertices $a$ and $b$, they have the same height if and only if $w_{a}=w_{b}$.
\end{cor}
\begin{proof}
Since $w_{a}=k^{h_{a}}$ and $w_{b}=k^{h_{b}}$ in this case.
\end{proof}

About the heights of the vertices we have the following proposition and corollary.
\begin{prop}\label{height}
Let $a\in T_{c}$, ${\rm ord}(c)=d|t$ and $h\ge 0$. Then ${\rm ord}(a)|k^{h}d$ if and only if $a\in T_{c}^{m}$, for some $m\le h$.
\end{prop}
\begin{proof}
Suppose ${\rm ord}(a)|k^{h}d$. Then $(a^{k^{h}})^{d}=1$, which implies ${\rm ord}(a^{k^{h}})|d$.
So $a^{k^{h}}$ is a cycle vertex. Hence, there exists $m\le h$ such that $a\in T_{c}^{m}$.

Conversely, suppose there exists $m\le h$ such that $a\in T_{c}^{m}$. Then $a^{k^{m}}=c$.
So $(a^{k^{m}})^{d}=c^{d}=1$, which implies ${\rm ord}(a)|k^{m}d$. Hence, ${\rm ord}(a)|k^{h}d$.
\end{proof}

\begin{cor}\label{level}
Let $a\in T_{c}$, ${\rm ord}(c)=d|t$ and $m\ge 1$. Then $a\in T_{c}^{m}$ if and only if ${\rm ord}(a)|k^{m}d$ and ${\rm ord}(a)\nmid k^{m-1}d$.
\end{cor}

For any $d\ge 1$, let $H_{d}$ be the subgroup of $H$ defined by $H_{d}=\{x\in H|x^{d}=1\}$.
It is well-known that $H_{d}$ is cyclic with order $(n,d)$. By Proposition \ref{vertex}, all cycle vertices of $G(n,k)$ form
the subgroup $H_{t}$. By Lemma \ref{solution}, all vertices with non-zero indegree form the subgroup $H_{\frac{n}{(k,n)}}$.
\begin{cor}\label{subgroup}
For any $d|t$ and $h\ge 0$, $\bigcup\limits_{\substack{0\le m \le h\\c\in T^{0},\, {\rm ord}(c)|d}}T_{c}^{m}$ is exactly the subgroup $H_{k^{h}d}$.
\end{cor}
\begin{proof}
By Proposition \ref{height} and the first part of its proof, this union consists of all $a\in H$ with
${\rm ord}(a)|k^{h}d$ . So it is exactly the subgroup $H_{k^{h}d}$.
\end{proof}

\begin{cor}
For any $l\ge 1$ and $h\ge 0$, $\bigcup\limits_{\substack{0\le m \le h\\c\in T^{0},\, \ell(c)|l}}T_{c}^{m}$ is exactly the subgroup $H_{k^{h}\cdot(t,k^{l}-1)}$.
\end{cor}
\begin{proof}
Since $c$ is a cycle vertex, ${\rm ord}(c)|t$. Then $\ell(c)|l$ if and only if $c^{k^{l}}=c$, that is ${\rm ord}(c)|(k^{l}-1)$, that is ${\rm ord}(c)|(t,k^{l}-1)$,
then applying Corollary \ref{subgroup}.
\end{proof}

For any set $X$, denote the number of its elements by $|X|$.
\begin{prop}\label{number}
For any cycle vertex $c$, we have:

$(1)$ $|T_{c}|=w$.

$(2)$ For $m\ge 1$, $|T^{m}|=(n,k^{m}t)-(n,k^{m-1}t)$ and $|T_{c}^{m}|=(w,k^{m})-(w,k^{m-1})$.

$(3)$ If $h_{0}\ge 2$, for $1\le m\le h_{0}-1$, the number of vertices in $T_{c}^{m}$ with indegree 0 is
$|T_{c}^{m}|-\frac{|T_{c}^{m+1}|}{(k,n)}$.
\end{prop}
\begin{proof}
(1) Note that there are $t$ cycle vertices and $n=wt$, by Corollary \ref{iso}, we have $|T_{c}|=w$.

(2) In Corollary \ref{subgroup}, fix $d=t$, put $h=m$ and $h=m-1$ respectively, we have
$T^{m}=\bigcup\limits_{\rm{ord}(c)|t}T_{c}^{m}=H_{k^{m}t}\setminus H_{k^{m-1}t}$.
So $|T^{m}|=(n,k^{m}t)-(n,k^{m-1}t)$. Since $|T_{c}^{m}|=\frac{1}{t}|T^{m}|$, we get the other formula.

(3) By Lemma \ref{solution}, the number of vertices in $T_{c}^{m}$ with non-zero indegree is $\frac{|T_{c}^{m+1}|}{(k,n)}$.
\end{proof}

Hence, if the unique cycle in a component of $G(n,k)$ has length $r$, then this component has $rw$ vertices.

\begin{cor}
$|T^{h_{0}}|\ge \frac{n}{2}$.
\end{cor}
\begin{proof}
Recall that $h_{0}$ is the height of the trees.

If $h_{0}=0$, then all vertices are in cycles, so $|T^{h_{0}}|=n \ge \frac{n}{2}$.

If $h_{0}\ge 1$, From Proposition \ref{number} (2), we have $|T^{h_{0}}|=n-(n,k^{h_{0}-1}t)=n-t(w,k^{h_{0}-1})$.
Since $n\nmid k^{h_{0}-1}t$, $w\nmid k^{h_{0}-1}$, which implies $(w,k^{h_{0}-1})\le \frac{w}{2}$.
Hence, we have $|T^{h_{0}}|\ge \frac{n}{2}$.
\end{proof}

In fact, the lower bound in the above corollary is the best one. For example, let $k=6$ and $n=2^{m}$,
where $m\ge 3$, then $t=1$ and $h_{0}=m$, so $|T^{h_{0}}|=n-(n,k^{h_{0}-1}t)=\frac{n}{2}$.

\begin{prop}
If $n\ge 5$ and $n$ is even, then the length of the longest cycle in $G(n,k)$ is less than or equal to $\frac{n-2}{2}$.
\end{prop}
\begin{proof}
If $(n,k)\ne 1$, then $h_{0}\ge 1$. By the above corollary, the number of non-cycle vertices is more than or
equal to $\frac{n}{2}$, which implies the number of cycle vertices is less than or equal to $\frac{n}{2}$.
Since the identity element of $H$ is in a loop, the length of the longest cycle in $G(n,k)$ is less than
or equal to $\frac{n}{2}-1=\frac{n-2}{2}$.

If $(n,k)=1$, then all vertices are in cycles and $t=n$. Notice that the length of the longest cycle is
$\ell(t)=\ell(n)={\rm ord}_{n}k$.
We factor $n$ as $2^{r}s$, where $r\ge 1$ and $(2,s)=1$. If $s\ne 1$,
then $\ell(n)\le \varphi(n)=2^{r-1}\varphi(s)< 2^{r-1}s = \frac{n}{2}$.
Since $\ell(n)$ is an integer, $\ell(n)\le \frac{n}{2}-1=\frac{n-2}{2}$.
If $s=1$, that is $n=2^{r}$, since $n\ge 5$, $r>2$, which implies $(\mathbb{Z}/n\mathbb{Z})^{*}$ has no primitive roots,
so $\ell(n)< \varphi(n)=2^{r-1}= \frac{n}{2}$, then $\ell(n)\le \frac{n-2}{2}$.
\end{proof}

Hence, the number of vertices in the largest component is less than or equal to $\frac{n-2}{2}w$.
In fact, the upper bound in the above proposition is the best one. For example, let $k=2$ and $n=2p$, $p$ is an odd prime,
and $2$ is the primitive root of $(\mathbb{Z}/p\mathbb{Z})^{*}$, then $t=p$ and $\ell(t)={\rm ord}_{p}2=p-1=\frac{n-2}{2}$.

\section{The adjacency matrix of $G(n,k)$}

For any two vertices $u$ and $v$ of $G(n,k)$, if $u^{k}=v$, we call $u$ a child of $v$.

If the vertex-set of $G(n,k)$ is $\{v_{1},v_{2},\cdots,v_{n}\}$, then the adjacency matrix of $G(n,k)$
is a $n\times n\, (0,1)$-matrix with the $(i,j)$-entry equal to the number of directed edges from $v_{i}$ to $v_{j}$,
we denote it by $A(n,k)$.

We label the vertices of $G(n,k)$ as follows. First, we label the vertices component by component, so we can get a block diagonal matrix.
Second, for each component, we label its vertices height by height according to the child relations.
For example, see Fig. 2 in \cite{VS}, let $H=(\mathbb{Z}/29\mathbb{Z})^{*}$ and $k=2$, then there are three components, see Figure 1.

\begin{figure}[h]
\begin{center}
\setlength{\unitlength}{1cm}
\begin{picture}(20,8)

\multiput(0.65,0)(1.8,0){3}{\circle{0.7}}
\put(0.55,-0.15){$7$}
\put(2.25,-0.15){$20$}
\put(4.05,-0.15){$23$}
\multiput(1,0)(1.8,0){2}{\vector(2,0){1.1}}
\put(2.45,-0.35){\oval(3.6,0.8)[b]}
\put(0.65,-0.45){\vector(0,1){0.1}}
\multiput(0.65,1.3)(1.8,0){3}{\circle{0.7}}
\multiput(0.65,0.95)(1.8,0){3}{\vector(0,-1){0.6}}
\put(0.55,1.2){$6$}
\put(2.25,1.2){$22$}
\put(4.15,1.2){$9$}
\multiput(0.2,2.6)(0.9,0){6}{\circle{0.7}}
\multiput(0.2,2.25)(1.8,0){3}{\vector(2,-3){0.4}}
\multiput(1.1,2.25)(1.8,0){3}{\vector(-2,-3){0.4}}
\put(0.1,2.5){$8$}
\put(0.92,2.5){$21$}
\put(1.82,2.5){$14$}
\put(2.72,2.5){$15$}
\put(3.68,2.5){$3$}
\put(4.52,2.5){$26$}

\multiput(6.85,0)(1.8,0){3}{\circle{0.7}}
\put(6.65,-0.15){$16$}
\put(8.45,-0.15){$24$}
\put(10.25,-0.15){$25$}
\multiput(7.2,0)(1.8,0){2}{\vector(2,0){1.1}}
\put(8.65,-0.35){\oval(3.6,0.8)[b]}
\put(6.86,-0.45){\vector(0,1){0.1}}
\multiput(6.85,1.3)(1.8,0){3}{\circle{0.7}}
\multiput(6.85,0.95)(1.8,0){3}{\vector(0,-1){0.6}}
\put(6.75,1.2){$4$}
\put(8.45,1.2){$13$}
\put(10.35,1.2){$5$}
\multiput(6.4,2.6)(0.9,0){6}{\circle{0.7}}
\multiput(6.4,2.25)(1.8,0){3}{\vector(2,-3){0.4}}
\multiput(7.3,2.25)(1.8,0){3}{\vector(-2,-3){0.4}}
\put(6.3,2.5){$2$}
\put(7.12,2.5){$27$}
\put(7.97,2.5){$10$}
\put(8.87,2.5){$19$}
\put(9.83,2.5){$11$}
\put(10.67,2.5){$18$}

\put(5.55,4){\circle{0.7}}    
\put(5.45,3.85){$1$}
\put(5.55,3.75){\oval(0.5,0.8)[b]}
\put(5.3,3.65){\vector(0,1){0.1}}
\put(5.55,5.3){\circle{0.7}}
\put(5.55,4.95){\vector(0,-1){0.6}}
\put(5.35,5.2){$28$}
\put(5.1,6.25){\vector(2,-3){0.4}}
\put(6.0,6.25){\vector(-2,-3){0.4}}
\multiput(5.1,6.6)(0.9,0){2}{\circle{0.7}}
\put(4.9,6.5){$12$}
\put(5.8,6.5){$17$}
\end{picture}
\end{center}
\qquad
\caption{The digraph $G(28,2)$}\label{h12}
\end{figure}

We label $G(28,2)$ by $v_{1}=1, v_{2}=28,v_{3}=12,v_{4}=17, v_{5}=7,v_{6}=20,v_{7}=23,v_{8}=6,v_{9}=22,v_{10}=9,v_{11}=8,
v_{12}=21,v_{13}=14,v_{14}=15,v_{15}=3,v_{16}=26, v_{17}=16,v_{18}=24,v_{19}=25,v_{20}=4,v_{21}=13,v_{22}=5,
v_{23}=2,v_{24}=27,v_{25}=10,v_{26}=19,v_{27}=11$ and $v_{28}=18$. Then digraph $G(28,2)$ is given in Figure 2.

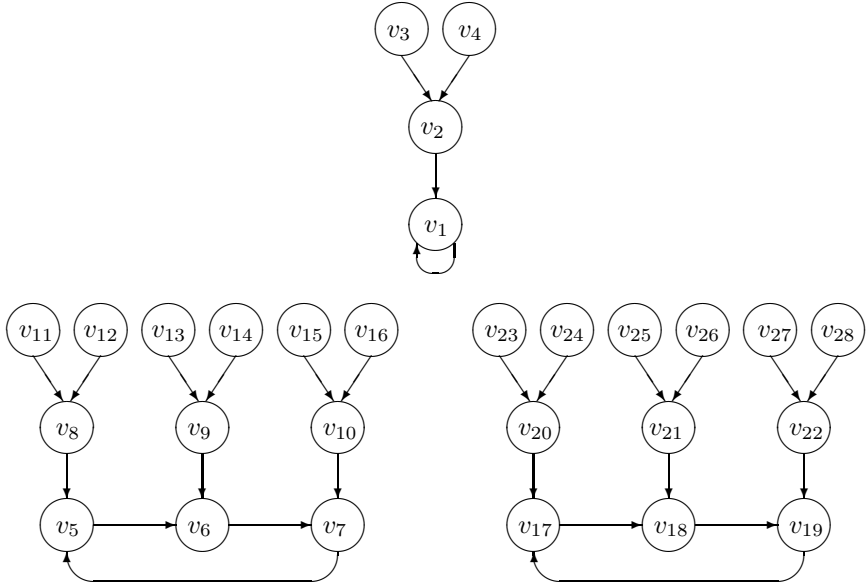
\begin{figure}[h]
\begin{center}
\setlength{\unitlength}{1cm}
\begin{picture}(20,8)

\multiput(0.65,0)(1.8,0){3}{\circle{0.7}}
\put(0.5,-0.1){$v_{5}$}
\put(2.25,-0.1){$v_{6}$}
\put(4.05,-0.1){$v_{7}$}
\multiput(1,0)(1.8,0){2}{\vector(2,0){1.1}}
\put(2.45,-0.35){\oval(3.6,0.8)[b]}
\put(0.65,-0.45){\vector(0,1){0.1}}
\multiput(0.65,1.3)(1.8,0){3}{\circle{0.7}}
\multiput(0.65,0.95)(1.8,0){3}{\vector(0,-1){0.6}}
\put(0.5,1.2){$v_{8}$}
\put(2.25,1.2){$v_{9}$}
\put(4.05,1.2){$v_{10}$}
\multiput(0.2,2.6)(0.9,0){6}{\circle{0.7}}
\multiput(0.2,2.25)(1.8,0){3}{\vector(2,-3){0.4}}
\multiput(1.1,2.25)(1.8,0){3}{\vector(-2,-3){0.4}}
\put(0.01,2.5){$v_{11}$}
\put(0.87,2.5){$v_{12}$}
\put(1.77,2.5){$v_{13}$}
\put(2.67,2.5){$v_{14}$}
\put(3.6,2.5){$v_{15}$}
\put(4.47,2.5){$v_{16}$}

\multiput(6.85,0)(1.8,0){3}{\circle{0.7}}
\put(6.65,-0.1){$v_{17}$}
\put(8.45,-0.1){$v_{18}$}
\put(10.25,-0.1){$v_{19}$}
\multiput(7.2,0)(1.8,0){2}{\vector(2,0){1.1}}
\put(8.65,-0.35){\oval(3.6,0.8)[b]}
\put(6.86,-0.45){\vector(0,1){0.1}}
\multiput(6.85,1.3)(1.8,0){3}{\circle{0.7}}
\multiput(6.85,0.95)(1.8,0){3}{\vector(0,-1){0.6}}
\put(6.65,1.2){$v_{20}$}
\put(8.39,1.2){$v_{21}$}
\put(10.25,1.2){$v_{22}$}
\multiput(6.4,2.6)(0.9,0){6}{\circle{0.7}}
\multiput(6.4,2.25)(1.8,0){3}{\vector(2,-3){0.4}}
\multiput(7.3,2.25)(1.8,0){3}{\vector(-2,-3){0.4}}
\put(6.2,2.5){$v_{23}$}
\put(7.07,2.5){$v_{24}$}
\put(7.97,2.5){$v_{25}$}
\put(8.87,2.5){$v_{26}$}
\put(9.83,2.5){$v_{27}$}
\put(10.67,2.5){$v_{28}$}

\put(5.55,4){\circle{0.7}}  
\put(5.4,3.9){$v_{1}$}
\put(5.55,3.75){\oval(0.5,0.8)[b]}
\put(5.3,3.65){\vector(0,1){0.1}}
\put(5.55,5.3){\circle{0.7}}
\put(5.55,4.95){\vector(0,-1){0.6}}
\put(5.35,5.2){$v_{2}$}
\put(5.1,6.25){\vector(2,-3){0.4}}
\put(6.0,6.25){\vector(-2,-3){0.4}}
\multiput(5.1,6.6)(0.9,0){2}{\circle{0.7}}
\put(4.9,6.5){$v_{3}$}
\put(5.85,6.5){$v_{4}$}

\end{picture}
\end{center}
\qquad
\caption{The digraph $G(28,2)$}\label{h12}
\end{figure}

If we partition $A(28,2)$ according to the components, then we can get a block diagonal matrix
and the main diagonal blocks are square matrixes. The main diagonal blocks are given as follows.
\begin{equation}
\begin{matrix}
B_{1}=\left(
\begin{smallmatrix}
&v_{1}&v_{2}&v_{3}&v_{4}\\
v_{1}&1 &0 &0 &0   \\
v_{2}&1 &0 &0 &0  \\
v_{3}&0 &1 &0 &0   \\
v_{4}&0 &1 &0 &0
\end{smallmatrix}\right),&
B_{2}=\left(
\begin{smallmatrix}
&v_{5}&v_{6}&v_{7}&v_{8}&v_{9}&v_{10}&v_{11}&v_{12}&v_{13}&v_{14}&v_{15}&v_{16}\\
v_{5}&0 &1 &0 &0 &0 &0 &0 &0 &0 &0 &0 &0   \\
v_{6}&0 &0 &1 &0 &0 &0 &0 &0 &0 &0 &0 &0   \\
v_{7}&1 &0 &0 &0 &0 &0 &0 &0 &0 &0 &0 &0   \\
v_{8}&1 &0 &0 &0 &0 &0 &0 &0 &0 &0 &0 &0   \\
v_{9}&0 &1 &0 &0 &0 &0 &0 &0 &0 &0 &0 &0   \\
v_{10}&0 &0 &1 &0 &0 &0 &0 &0 &0 &0 &0 &0   \\
v_{11}&0 &0 &0 &1 &0 &0 &0 &0 &0 &0 &0 &0    \\
v_{12}&0 &0 &0 &1 &0 &0 &0 &0 &0 &0 &0 &0  \\
v_{13}&0 &0 &0 &0 &1 &0 &0 &0 &0 &0 &0 &0   \\
v_{14}&0 &0 &0 &0 &1 &0 &0 &0 &0 &0 &0 &0   \\
v_{15}&0 &0 &0 &0 &0 &1 &0 &0 &0 &0 &0 &0   \\
v_{16}&0 &0 &0 &0 &0 &1 &0 &0 &0 &0 &0 &0    \\
\end{smallmatrix}\right),
\end{matrix}
\notag
\end{equation}

\begin{equation}
B_{3}=\left(
\begin{smallmatrix}
&v_{17}&v_{18}&v_{19}&v_{20}&v_{21}&v_{22}&v_{23}&v_{24}&v_{25}&v_{26}&v_{27}&v_{28}\\
v_{17} &0 &1 &0 &0 &0 &0 &0 &0 &0 &0 &0 &0 \\
v_{18} &0 &0 &1 &0 &0 &0 &0 &0 &0 &0 &0 &0  \\
v_{19} &1 &0 &0 &0 &0 &0 &0 &0 &0 &0 &0 &0  \\
v_{20} &1 &0 &0 &0 &0 &0 &0 &0 &0 &0 &0 &0  \\
v_{21} &0 &1 &0 &0 &0 &0 &0 &0 &0 &0 &0 &0   \\
v_{22} &0 &0 &1 &0 &0 &0 &0 &0 &0 &0 &0 &0  \\
v_{23} &0 &0 &0 &1 &0 &0 &0 &0 &0 &0 &0 &0  \\
v_{24} &0 &0 &0 &1 &0 &0 &0 &0 &0 &0 &0 &0  \\
v_{25} &0 &0 &0 &0 &1 &0 &0 &0 &0 &0 &0 &0  \\
v_{26} &0 &0 &0 &0 &1 &0 &0 &0 &0 &0 &0 &0   \\
v_{27} &0 &0 &0 &0 &0 &1 &0 &0 &0 &0 &0 &0  \\
v_{28} &0 &0 &0 &0 &0 &1 &0 &0 &0 &0 &0 &0  \\
\end{smallmatrix}\right).
\notag
\end{equation}

Since $B_{2}$ and $B_{3}$ correspond to isomorphic components, $B_{2}=B_{3}$.
If we partition each main diagonal block according to the heights, then we can get a block lower triangular matrix
and its main diagonal blocks are square matrixes, its main diagonal blocks are all equal to $0$ except the $(1,1)$-block.
After partitioning, $B_{1},B_{2}$ and $B_{3}$ have the following form.
\begin{equation}
B_{i}=\begin{pmatrix}
B_{i0} &0 &0 \\
B_{i1} &0 &0 \\
0 &B_{i2} &0
\end{pmatrix},i=1,2,3.
\notag
\end{equation}

We denote the characteristic polynomial and minimal polynomial of a matrix $A$ by $f_{A}(\lambda)$ and $m_{A}(\lambda)$ respectively.
Notice that the characteristic polynomial and minimal polynomial of a $r\times r$ matrix with the following form
\begin{equation}\label{matrix1}
\begin{pmatrix}
0 &1 & \\
 &\ddots &\ddots \\
 & &\ddots&1\\
1 & &&0
\end{pmatrix}
\end{equation}
are both $\lambda^{r}-1$.
Hence, $f_{B_{1}}(\lambda)=\lambda^{3}(\lambda-1)$ and
$f_{B_{2}}(\lambda)=f_{B_{3}}(\lambda)=\lambda^{9}(\lambda^{3}-1)$.

\begin{lem}\label{minimal}
If a partitioned matrix $D$ has the following form
\begin{equation}
D=
\begin{pmatrix}
D_{0} & & \\
D_{1} &0 & \\
 &\ddots &\ddots&\\
 & &D_{m}&0
\end{pmatrix},
\notag
\end{equation}
where the main diagonal blocks are all square matrixes, $D_{0}$ is a $r\times r$ matrix with the form as $(\ref{matrix1})$, each $D_{i}$ $(1\le i\le m)$ is non-negative and its $(1,1)$-entry is positive. Then $m_{D}(\lambda)=\lambda^{m}(\lambda^{r}-1)$.
\end{lem}
\begin{proof}
It is obvious that $\lambda^{r}-1|m_{D}(\lambda)$ and the first row of the partitioned matrix $D^{r}-I$ is zero,
where $I$ is the identity matrix.

Since
\begin{equation}
D^{m}=
\begin{pmatrix}
D_{0}^{m}&0&\cdots&0 \\
D_{1}D_{0}^{m-1} & 0&\cdots&0 \\
D_{2}D_{1}D_{0}^{m-2} &0&\cdots&0  \\
 \vdots& \vdots&&\vdots\\
D_{m}D_{m-1}\times\cdots \times D_{1}&0&\cdots&0
\end{pmatrix},
\notag
\end{equation}
$D^{m}(D^{r}-I)=0$.

Since
\begin{equation}
D^{m-1}=
\begin{pmatrix}
D_{0}^{m-1}&0 &\cdots & 0\\
D_{1}D_{0}^{m-2} & 0&\cdots& 0\\
\vdots& \vdots&&\vdots\\
D_{m-2}D_{m-3}\times \cdots\times D_{0} &0 &\cdots&0\\
D_{m-1}D_{m-2}\times\cdots \times D_{1}&0&\cdots&0\\
0&D_{m}D_{m-1}\times\cdots \times D_{2}&\cdots&0\\
\end{pmatrix},
\notag
\end{equation}
the $(m+1,1)$-entry of the partitioned matrix $D^{m-1}(D^{r}-I)$ is $D_{m}D_{m-1}\times\cdots \times D_{2}\times D_{1}D_{0}^{r-1}$.
Since $D_{0}^{r-1}$ is invertible and non-negative, there exists a positive entry in the first row of $D_{0}^{r-1}$.
Notice that $D_{m}D_{m-1}\times\cdots \times D_{2}D_{1}$ is non-negative and its $(1,1)$-entry is positive. Hence,
$D_{m}D_{m-1}\times\cdots \times D_{2}\times D_{1}D_{0}^{r-1}\ne 0$. So $D^{m-1}(D^{r}-I)\ne 0$.

Hence, we have $m_{D}(\lambda)=\lambda^{m}(\lambda^{r}-1)$.
\end{proof}

So by Lemma \ref{minimal}, $m_{B_{1}}(\lambda)=\lambda^{2}(\lambda-1)$ and $m_{B_{2}}(\lambda)=m_{B_{3}}(\lambda)=\lambda^{2}(\lambda^{3}-1)$.

Recall that $h_{0}$ is the height of the trees. Let $C$ be a component of $G(n,k)$
and the unique cycle in $C$ has length $r$, then $C$ has $rw$ vertices,
the characteristic polynomial and minimal polynomial of $C$ is $\lambda^{rw-r}(\lambda^{r}-1)$ and $\lambda^{h_{0}}(\lambda^{r}-1)$
respectively.

Suppose the components of $G(n,k)$ consist of $m_{1}$ copies of $G_{1}$, $m_{2}$ copies of $G_{2}$, $\cdots$,
$m_{s}$ copies of $G_{s}$, where $G_{1}, G_{2},\cdots, G_{s}$ are pairwise non-isomorphic, the unique cycle in each $G_{i}$ $(1\le i\le s)$
has length $r_{i}$.
Then we get the following theorem.
\begin{thm}
$(1)$ The characteristic polynomial of $G(n,k)$ is $\prod\limits_{i=1}^{s}\big[\lambda^{r_{i}w-r_{i}}(\lambda^{r_{i}}-1)\big]^{m_{i}}$.

$(2)$ The minimal polynomial of $G(n,k)$ is $\lambda^{h_{0}}(\lambda^{\ell(t)}-1)$.
\end{thm}
\begin{proof}
The result in $(1)$ is obvious.

By Proposition \ref{vertex} and Proposition \ref{cycle} $(1)$, the length of each cycle divides $\ell(t)$, this yields the result in $(2)$.
\end{proof}

By the discussions in section 2 and section 3, if we specify the values of $n$ and $k$, we can calculate explicitly these data $s,t,w,h_{0}, \ell(t),
m_{i}$ and $r_{i}$ $(1\le i\le s)$.

Since we have determined the characteristic polynomial of $G(n,k)$, it is easy to get the eigenvalues and spectrum of $G(n,k)$.

\section{The Automorphism Group of $G(n,k)$}
For any graph $G$, we denote its automorphism group by Aut$(G)$.
For simplicity, we denote the automorphism group of $G(n,k)$ by Aut$(n,k)$.
Notice that Aut$(G)$ is a permutation group on $\{1,2,\cdots,|G|\}$.

Let $S_{m}$ be the symmetric group on $\{1,2,\cdots,m\}$.
Let $P_{1}$ and $P_{2}$ be two permutation groups on $\{1,2,\cdots, m\}$ and $\{1,2,\cdots, r\}$ respectively.
Recall that the wreath product $P_{1}\wr P_{2}$ is generated by the direct product of $r$ copies of $P_{1}$,
together with the elements of $P_{2}$ acting on these $r$ copies of $P_{1}$.

Using the notations in the above section, we get the following theorem.
\begin{thm}
{\rm Aut}$(n,k)\cong ({\rm Aut}(G_{1})\wr S_{m_{1}})\times ({\rm Aut}(G_{2})\wr S_{m_{2}})\times \cdots \times ({\rm Aut}(G_{s})\wr S_{m_{s}})$.
\end{thm}
\begin{proof}
See Theorem 1.1 in \cite{C}.
\end{proof}

For each component $G_{i}$ $(1\le i \le s)$, its unique cycle has length $r_{i}$, by Corollary \ref{iso}, we have
the following proposition.
\begin{prop}
For each $1\le i \le s$, ${\rm Aut}(G_{i})\cong {\rm Aut}(T_{1})\wr <\sigma_{i}>$, where $\sigma_{i}$ is a $r_{i}$-cycle,
\begin{equation}
\sigma_{i}=
\begin{pmatrix}
1&2&3&\cdots &r_{i}\\
2&3&4&\cdots &1
\end{pmatrix}.
\notag
\end{equation}
\end{prop}
\begin{proof}
Notice that the automorphism group of the cycle in $G_{i}$ is exactly the permutation group generated by $\sigma_{i}$.
\end{proof}

Hence, we only need to determine ${\rm Aut}(T_{1})$. If $(n,k)=1$, by Proposition \ref{regular},
we have ${\rm Aut}(T_{1})=\{1\}$. Then we get the following proposition.
\begin{prop}\label{auto1}
If $(n,k)=1$, then {\rm Aut}$(n,k)=(<\sigma_{1}>\wr S_{m_{1}})\times (<\sigma_{2}>\wr S_{m_{2}})\times \cdots \times (<\sigma_{s}>\wr S_{m_{s}})$.
\end{prop}

\begin{prop}
If $k=1$, then {\rm Aut}$(n,k)=S_{n}$.
\end{prop}

\begin{prop}
If $k=n$, then {\rm Aut}$(n,k)=S_{n-1}$.
\end{prop}

But in general it is difficult to determine ${\rm Aut}(T_{1})$.
Since the vertices with the same height may have different number of children. For example,
let $H=(\mathbb{Z}/41\mathbb{Z})^{*}$ and $k=4$, $T_{1}$ is given as follows.
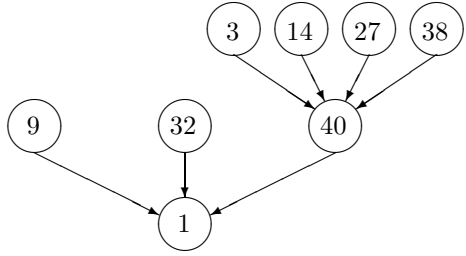
\begin{figure}[h]
\begin{center}
\setlength{\unitlength}{1cm}
\begin{picture}(20,3)
\put(4.85,0){\circle{0.7}}
\put(4.75,-0.1){$1$}
\multiput(2.85,1.3)(2,0){3}{\circle{0.7}}
\put(2.75,1.2){$9$}
\put(4.65,1.2){$32$}
\put(6.65,1.2){$40$}
\put(2.85, 0.95){\vector(2,-1){1.66}}
\put(4.85, 0.95){\vector(0,-1){0.6}}
\put(6.85, 0.95){\vector(-2,-1){1.66}}
\multiput(5.5,2.6)(0.9,0){4}{\circle{0.7}}
\put(5.4,2.45){$3$}
\put(6.2,2.45){$14$}
\put(7.1,2.45){$27$}
\put(8.0,2.45){$38$}
\put(5.5, 2.25){\vector(3,-2){1.075}}
\put(8.2, 2.25){\vector(-3,-2){1.075}}
\put(6.4, 2.25){\vector(1,-2){0.31}}
\put(7.3, 2.25){\vector(-1,-2){0.31}}
\end{picture}
\end{center}
\caption{The tree $T_{1}$ for $n=40$ and $k=4$}\label{h12}
\end{figure}

Recall that if $h_{0}$ is the height of $T_{1}$, then the vertices of $T_{1}$ form the subgroup $H_{k^{h_{0}}}$,
that is $H_{w}$. So ${\rm Aut}(H_{w})\subseteq {\rm Aut}(T_{1})$.

As follows we want to determine ${\rm Aut}(T_{1})$ when $k$ is a prime.

For any two vertices $a$ and $b$, if there is a $g\in {\rm Aut}(n,k)$ such that $g(a)=b$, we say $a$ is isomorphic to $b$,
denote it by $a\cong b$.
This is an equivalent relation in $G(n,k)$. We will show that if ${\rm ord}(a)={\rm ord}(b)$, then $a\cong b$.

Suppose that $M$ is a cyclic group with $m$ elements. Given three positive integers $r$, $r_{1}$ and $q$ such that $r|r_{1}|m$
and $r_{1}=rq$.  For any $b\in M, {\rm ord}(b)=r$, put $M_{b}=\{a\in M|a^{q}=b,{\rm ord}(a)=r_{1}\}$.
Then we have the following lemma.

\begin{lem}\label{solution2}
For any $b\in M$ with ${\rm ord}(b)=r$, $|M_{b}|=\frac{\varphi(r_{1})}{\varphi(r)}$.
\end{lem}
\begin{proof}
Fix a generator $\zeta$ of $M$ such that $\zeta^{\frac{m}{r}}=b$.
It is easy to see that $\zeta^{\frac{m}{r_{1}}}\in M_{b}$. So $M_{b}$ is not empty.

Every element with order $r_{1}$ has a unique form $\zeta^{\frac{ml_{1}}{r_{1}}}, 1\le  l_{1}\le r_{1}, (l_{1},r_{1})=1$. Then we have
\begin{equation}
\zeta^{(\frac{ml_{1}}{r_{1}})q}=b \Leftrightarrow m|(\frac{ml_{1}}{r}-\frac{m}{r}) \Leftrightarrow r|l_{1}-1.
\notag
\end{equation}
So $|M_{b}|=|\big\{l_{1}|1\le l_{1}\le r_{1}, (l_{1},r_{1})=1,r|l_{1}-1\big\}|$, which implies that
$|M_{b}|$ only depends on $r$ and $r_{1}$ and it is independent of the specified value of $b$.
Hence, given another $b^{'}\in M, {\rm ord}(b^{'})=r$, we have $|M_{b^{'}}|=|M_{b}|$.

Since there are $\varphi(r_{1})$ elements with order $r_{1}$ and $\varphi(r)$ elements with order $r$,
$|M_{b}|=\frac{\varphi(r_{1})}{\varphi(r)}$.
\end{proof}

\begin{cor}\label{solution3}
For any two elements $a,b\in M$, ${\rm ord}(a)={\rm ord}(b), q\ge 1$,
then for each positive integer $r$ such that ${\rm ord}(a)|r$,
$M_{1}=\{x|x^{q}=a\}$ and $M_{2}=\{x|x^{q}=b\}$ have the same number of elements with order $r$.
\end{cor}
\begin{proof}
By Lemma \ref{solution}, $M_{1}$ and $M_{2}$ have the same number of elements.
Then we can get the desired result by applying Lemma \ref{solution2}.
\end{proof}

\begin{thm}\label{iso2}
For any $a,b\in G(n,k)$, if ${\rm ord}(a)={\rm ord}(b)$, then $a\cong b$.
\end{thm}
\begin{proof}
By Corollary \ref{height2}, $a$ and $b$ are at the same height. Since for any positive integer $h$,
${\rm ord}(a^{h})={\rm ord}(b^{h})$, then the cycles which they lead to have the same order.
Then the desired result follows from Corollary \ref{solution3}.
\end{proof}

From now on we assume that $k$ is a prime.

For any $a\in T_{1}$, there exists a $h$ such that $a^{k^{h}}=1$, which implies that ${\rm ord}(a)|k^{h}$.
So $w_{a}={\rm ord}(a)$. By Corollary \ref{prime}, we get the following proposition.
\begin{prop}\label{T1}
If $k$ is a prime, for any $a,b\in T_{1}$, $a$ and $b$ are at the same height if and only if ${\rm ord}(a)={\rm ord}(b)$.
\end{prop}

Hence, all the vertices with indegree $0$ of $T_{1}$ are at the largest height $h_{0}$.

Since $k$ is a prime, $(n,k)=1$ or $k$. We have discussed Aut$(n,k)$ on the case $(n,k)=1$, see Proposition \ref{auto1}.

As follows we suppose that $(n,k)=k$. Then the largest height $h_{0}\ge 1$.
For $1\le h\le h_{0}$, let $T_{1h}$ be the tree originating from a vertex with height $h$ in $T_{1}$.
In particular, the vertex set of $T_{1h_{0}}$ contains only one point.
Proposition \ref{T1} and Theorem \ref{iso2} tell us that $T_{1h}$ is well-defined.
Then we get the following proposition.
\begin{prop}
If $k$ is a prime and $(n,k)=k$, then we have
${\rm Aut}(T_{1})\cong {\rm Aut}(T_{11})\wr S_{k-1}$, for any $1\le h< h_{0}$, ${\rm Aut}(T_{1h})={\rm Aut}(T_{1,h+1})\wr S_{k}$,
and ${\rm Aut}(T_{1h_{0}})=\{1\}$.
\end{prop}

\section{Further Problems}
We mention three further problems which may worth studying.

First, it may be interesting to consider other graphic problems for $G(n,k)$,
such as the matching problem and the coloring problem.

Second, it may be interesting to study the asymptotic mean numbers of cycle vertices and cycles.
\cite{CS}, \cite{SH} and \cite{VS} will be helpful.

Third, what will happen if $H$ is not cyclic?
\cite{BHM}, \cite{SK1}, \cite{SK2}, \cite{SK3} and \cite{W}
will be helpful.

\section{Acknowledgment}
We would like to thank Dr. Jingfen Lan for her valuable suggestions.
We also thank the referee for the careful review and the valuable comments.








\end{document}